\def\timestamp{%
Time-stamp: <library.tex: maandag 18-08-2025 at 17:31:01 (cest)>}
\def\stripname Time-stamp: <#1: #2 #3 at #4 #5>{#3/#4 (#1)}
\edef\filedate{\expandafter\stripname\timestamp}
\documentclass[a4paper]{amsart}

\usepackage{amsrefs}

\newcommand\orpr[2]{\langle{#1},{#2}\rangle}
\newcommand\norm[1]{\mathopen\|#1\mathclose\|}
\newcommand\abs[1]{\mathopen|#1\mathclose|}
\newcommand\preim{^\gets}
\DeclareMathSymbol\N0{AMSb}{`N}
\DeclareMathSymbol\Irr0{AMSb}{`P}
\DeclareMathSymbol\Q0{AMSb}{`Q}
\DeclareMathSymbol\R0{AMSb}{`R}
\DeclareMathSymbol\restr\mathbin{AMSa}{"16}
\DeclareMathSymbol\le    \mathrel{AMSa}{"36}
\DeclareMathSymbol\ge    \mathrel{AMSa}{"3E}
\newcommand\Open{\operatorname{O}}
\newcommand\Coz{\operatorname{Coz}}

\theoremstyle{plain}

\newtheorem{lemma}{Lemma}[section]
\newtheorem{theorem}[lemma]{Theorem}
\newtheorem{proposition}[lemma]{Proposition}

\theoremstyle{remark}

\newtheorem{remark}[lemma]{Remark}

\begin{document}

\title{An infinite library}
\author[K. P. Hart]{Klaas Pieter Hart}
\address{Faculty EEMCS\\TU Delft\\
         Postbus 5031\\2600~GA {} Delft\\the Netherlands}
\email{k.p.hart@tudelft.nl}
\urladdr{https://fa.ewi.tudelft.nl/\~{}hart}

\begin{abstract}
The purpose of this note is to describe a space that is regular but not 
completely regular, but only barely so: all closed sets are $G_\delta$-sets
and every singleton is a zero-set.  
\end{abstract}

\date\filedate

\maketitle

\section*{Introduction}

There are many examples of regular spaces that are not completely regular.
A particularly simple one is due to Adam Mysior, see~\cite{MR0601748}
or \cite{MR1039321}*{Example~1.5.9}.

What many of the examples have in common is that in order to ensure 
their non-complete regularity most continuous real-valued functions are
constant on relatively large sets.
This is often ensured by having very few points of first-countability
and very few closed sets that are $G_\delta$-sets.

In this note we adapt a method due to F.~B. Jones from~\cite{MR0413044}
to produce an example that is first-countable and in which every closed
set is a $G_\delta$-set and every singleton set is a zero-set.
This method goes back to Tychonoff's paper~\cite{MR1512595} where
it is applied to the top and right-hand lines of the Tychonoff plank.

We apply it to the sets $Q$ and~$P$ of the rational and irrational points
of the $x$-axis in the Niemytzki plane respectively.

The note is organized as follows.
Section~\ref{sec.Niemytzki} contains a description of the Niemytzki plane~$N$
and an analysis of the behaviour of continuous real-valued functions on the
$x$-axis.
In Section~\ref{sec.book} we apply Jones' method to the Niemytzki plane
and obtain an analogue of Tychonoff's example, but with in which every
closed set is a $G_\delta$-set and in which all points, but one, are zero-sets;
in Section~\ref{sec.library} we show how to make that one point a zero-set 
as well.

\section{The Niemytzki plane and continuous functions}
\label{sec.Niemytzki}

Let $N$ denote the Niemytzki plane,
see~\cite{MR0345087}*{\S\,1, Nr.~6, $2^\circ$}, 
or \cite{MR1039321}*{Example~1.2.4},
or~\cite{MR507446}*{Example~82}.

The underlying set of~$N$ is the closed upper half-plane in~$\R^2$, that is,
$N=\{\orpr xy: y\ge0\}$.
The topology is defined by specifying local bases at all points of~$N$.

If $y>0$ then the $n$th basic neighbourhood of $\orpr xy$ is
$$
U(x,y,n)=\bigl\{\orpr uv\in N:\norm{\orpr uv-\orpr xy}<2^{-n}\bigr\}
$$
and $\{U(x,y,n)n\in\N\}$ is the local base at~$\orpr xy$.

If $y=0$ then we denote, for $n\in\N$, by $U'(x,n)$ the open disc
$$
\bigl\{\orpr uv: \norm{\orpr uv-\orpr x{2^{-n}}} <2^{-n}\bigr\}
$$
and the $n$th basic neighbourhood of~$\orpr x0$
then is $U(x,n)=U'(x,n)\cup\{\orpr x0\}$.

\smallskip
In~$N$ we have two nice disjoint closed sets that cannot be separated by
disjoint open sets nor, a fortiori, by continuous real-valued functions.
These sets are
\begin{itemize}
\item $P=\{\orpr x0:x$ is irrational$\}$, and
\item $Q=\{\orpr x0:x$ is rational$\}$.
\end{itemize}
In~\cite{MR0345087}*{\S\,5, Nr.~5} this is part of an exercise with a hint
in a footnote, which comes down to: ``Use the Baire Category theorem''.

\smallbreak
We need some quantitative information about how close these two sets are
in the context of continuous real-valued functions.
To this end we formulate a few lemmas.
The first two are quite elementary, but useful.

\begin{lemma}\label{lemma:seq}
Let $a\in\R$ and let $\langle a_k:k\in\N\}$ be a sequence in~$\R$ that 
converges to~$a$.
Then for every~$n$ we have $U'(a,n)\subseteq\bigcup_{k\in\N}U'(a_k,n)$.
\end{lemma}

\begin{proof}
Let $\orpr xy\in U'(a,n)$, then $\norm{\orpr xy-\orpr a{2^{-n}}}<2^{-n}$.  
Take $k\in\N$ such that 
$$
\abs{a_k-a}<2^{-n}-\norm{\orpr xy-\orpr a{2^{-n}}}
$$
Then $\norm{\orpr xy-\orpr{a_k}{2^{-n}}}<2^{-n}$, and so $\orpr xy\in U'(a_k,n)$.
\end{proof}

From this we deduce the following lemma.

\begin{lemma}\label{lemma:interval}
Let $I=[a,b]$ be a closed interval in~$\R$, let $D$ be dense in~$I$, 
and let~$n\in\N$.
Then $\bigcup_{d\in D}U'(d,n)=\bigcup_{x\in I}U'(x,n)$.  
\end{lemma}

\begin{proof}
One inclusion is clear.

For the other let $x\in I$ and let $\langle d_k:k\in\N\rangle$ be a sequence
in~$D$ that converges to~$x$. 
Then by Lemma~\ref{lemma:seq} we have 
$U'(x,n)\subseteq\bigcup_kU'(d_k,n)\subseteq\bigcup_{d\in D}U'(d,n)$.   
\end{proof}

We apply Lemma~\ref{lemma:interval} and the Baire Category Theorem to 
establish the following two lemmas.

\begin{lemma}\label{lemma.PtoQ}
If $f:N\to[0,1]$ is continuous and $G$~is a dense $G_\delta$-subset of\/~$\R$
such that $f\restr G\equiv0$ then for every open
interval~$I$ in~$\R$ we have $\inf\{f(q,0):q\in I\cap\Q\}=0$.
\end{lemma}

\begin{proof}
It suffices to show that for every nonempty open interval~$I$ and 
every $\varepsilon>0$ there is a $q\in\Q$ such that $f(q,0)<\varepsilon$.

Let $I$ and $\varepsilon$ be given.
For every $p\in G$ we let $n_p$ be the smallest natural number 
for which $f[U(p,n_p)]\subseteq[0,\frac12\varepsilon)$.

By the Baire Category Theorem there are a nonempty open subinterval~$J$ 
of~$I$ and a natural number~$n$ such that $J$~is a subset of the closure 
of $A=\{p:n_p=n\}$.
Now both $A\cap J$ and $\Q\cap J$ are dense in~$J$, 
and so by Lemma~\ref{lemma:interval} we have
$$
\bigcup_{p\in A\cap J} U'(p,n) = \bigcup_{q\in\Q\cap J} U'(q,n)
$$
This shows that $f[U'(q,n)]\subseteq[0,\frac12\varepsilon)$ for 
all~$q\in\Q\cap J$.
But then for all $q\in\Q\cap J$ we have
$f(q,0)\le\frac12\varepsilon<\varepsilon$.
\end{proof}

We also have the converse of this lemma.

\begin{lemma}\label{lemma.QtoP}
If $f:N\to[0,1]$ is continuous and such that for every nonempty
open interval~$I$ in~$\R$ we have $\inf\{f(q,0):q\in I\cap\Q\}=0$ then the set
$\{x\in\R:f(x,0)=0\}$ contains a dense $G_\delta$-set of~$\R$.
\end{lemma}

\begin{proof}
We show that for every~$k\in\N$ and every nonempty open interval~$I$ 
there is nonempty open subinterval~$J$ of~$I$ such that $f(x,0)<2^{-k}$ 
for all~$x\in J$.

This then implies that there is a sequence $\langle O_k:k\in\N\rangle$
of dense open sets in~$\R$ such that $f(x,0)<2^{-k}$ whenever
$x\in O_k$.
The intersection~$\bigcap_{k\in\N}O_k$ is a dense $G_\delta$-set and
contained in $\{x:f(x,0)=0\}$.

\smallskip
Let $k$ and $I$ be given.
Let $\varepsilon=\frac13\cdot2^{-k}$ and cover $[0,1]$ by the
intervals $K_i=\bigl(i\cdot\varepsilon,(i+2)\cdot\varepsilon\bigr)$, 
where $i=-1$, $0$, \dots, $3\cdot2^k-1$.

For $p\in I$ let $i_p$ be minimal with $f(p,0)\in K_{i_p}$ and let $n_p$ 
be minimal such that $f[U(p,n_p)]\subseteq K_{i_p}$.
By the Baire Category Theorem we have $i$ and $n$ and 
a nonempty open subinterval~$J$
of~$I$ such that $A=\{p:i_p=i$ and $n_p=n\}$ contains the closure of~$J$.

Again by Lemma~\ref{lemma:interval} we have 
$$
\bigcup_{p\in A\cap J} U'(p,n) = \bigcup_{q\in\Q\cap J} U'(q,n)
$$
This now implies that $f[U'(q,n)]\subseteq K_i$, and hence that
$i\cdot\varepsilon\le f(q,0)$, whenever $q\in\Q\cap J$.
But $\inf\{f(q,0):q\in\Q\cap J\}=0$ and hence $i=-1$ or $i=0$.

We find that $f[U(x,n)]\subseteq[0,2\varepsilon]$, and hence
$f(x,0)\le2\varepsilon<2^{-k}$, for all $x\in J$.
\end{proof}

\section{A Book}
\label{sec.book}

We apply Jones' method from~\cite{MR0413044} to~$N$ to obtain a 
regular, non-completely regular space, as follows.

We start with the product $N\times\N$, where $\N$ carries its discrete
topology.

For every even $n\in\N$ and $q\in\Q$ we identify the points 
$\orpr{\orpr q0}{n}$ and $\orpr{\orpr q0}{n+1}$, so that $Q\times\{n\}$
and~$Q\times\{n+1\}$ become one copy of~$Q$, which we denote~$Q_n$.

Likewise for every odd $n\in\N$ and $p\in\Irr$ we identify
$\orpr{\orpr p0}{n}$ and $\orpr{\orpr p0}{n+1}$, thus creating out of
$P\times\{n\}$ and~$P\times\{n+1\}$ one copy of~$P$ that we denote~$P_n$.

The resulting space we call $B$ and we let $\pi:N\times\N\to B$ denote 
the quotient map.

\begin{remark}
In \cite{MR507446}*{Examples 90 and~91} Tychonoff's example, mentioned
in the introduction, is represented pictorially as a corkscrew.

Our space $B$ looks more like a book made from infinitely many sheets
of paper sewn together alternately along the rational and irrational 
points on the $x$-axis, hence the title of the present section.
We shall call $B$ a \emph{book} from now on. 
\end{remark}

Lemmas~\ref{lemma.PtoQ} and~\ref{lemma.QtoP} above imply that 
if $f:B\to[0,1]$ is continuous and equal to~$0$
on the set~$Q_0$ then, by induction, the following holds for every~$n$,
where we let $F:N\times\N\to[0,1]$ be the composition $f\circ\pi$.
\begin{itemize}
\item if $n$ is even then $\inf\{F(x,0,n):x\in O\cap\Q\}=0$ 
      whenever $O$ is an open interval in~$\R$, and
\item if $n$ is odd then there is a dense $G_\delta$-set $G_n$, such that
      $F(x,0,n+1)=F(x,0,n)=0$ whenever $x\in G_n\cap P_n$.
\end{itemize}

\begin{remark}\label{rem:Gdelta}
Note that by the second item there is in fact a single dense $G_\delta$-set~$G$,
to wit~$\bigcap\{G_n:n$~is odd$\}$, such that 
$F(x,0,n+1)=F(x,0,n)=0$ whenever $x\in G\cap P_n$ and~$n$~is odd.
\end{remark}

This implies that if we were to add a `point at infinity' $\infty$
to $N\times\N$ with basic neighbourhoods 
$U_m=\{\infty\}\cup(N\times\{n\in\N:n\ge m\})$
and apply the quotient operation above to the new space
$(N\times\N)\cup\{\infty\}$,
then the resulting quotient space $B\cup\{\infty\}$ is not completely 
regular at~$\infty$.

\begin{lemma}\label{lemma:piclosed}
The map $\pi:(N\times\N)\cup\{\infty\}\to B\cup\{\infty\}$ is closed.  
\end{lemma}

\begin{proof}
If $F$~is closed in $(N\times\N)\cup\{\infty\}$ then 
$\pi\preim\bigl[\pi[F]\bigr]$~is closed as well.
If $n\in\N$ then its intersection with the sheet $N\times\{n\}$
consists of three parts (two if $n=0$): 
\begin{itemize}
\item $F\cap(N\times\{n\})$, 
\item the set $\{\orpr{\orpr x0}{n}: \orpr{\orpr x0}{n+1}\in F\}$, and
\item the set $\{\orpr{\orpr x0}{n}: \orpr{\orpr x0}{n-1}\in F\}$
if also $n>0$.
\end{itemize}
The union of these intersections is closed in $N\times\N$.

If $\infty\notin F$ then $F$~intersects only finitely many sheets and
so $\infty$~is not in the closure of $\pi\preim\bigl[\pi[F]\bigr]$.    

If $\infty\in F$ then all is well.
\end{proof}

\begin{lemma}[Exercise 1.5.H in \cite{MR1039321}]\label{lemma:Nperfect}
In the Niemytzki plane every closed set is a $G_\delta$-set.  
\end{lemma}

\begin{proof}[Solution]
Let $F$ be closed in~$N$ and divide it onto $F_1$, its intersection with the
$x$-axis, and $F_2=F\setminus F_1$.

The set $F_2$ is closed in the open subspace $\{\orpr xy:y>0\}$, which carries
its Euclidean topology.
Hence $F_2$~is a $G_\delta$-set in that subspace and hence in~$N$.

Let $A=\{x:\orpr x0\in F\}$; then 
$F_1=\bigcap_{k\in\N}\bigcup_{x\in A}U(x,k)$.

So $F$~is the union of two $G_\delta$-sets.
\end{proof}

\begin{proposition}\label{prop:Bperfect}
In the space $B\cup\{\infty\}$ every open set is an $F_\sigma$-set.  
\end{proposition}

\begin{proof}
Let $O$ be open in $B\cup\{\infty\}$, then $\pi\preim[O]$ is open
in $(N\times\N)\cup\{\infty\}$ and hence an $F_\sigma$-set: 
By Lemma~\ref{lemma:Nperfect} its intersection with every sheet 
$N\times\{n\}$ is an $F_\sigma$-set, 
and, if $\infty\in\pi\preim[O]$ then $\{\infty\}$~is a contributing
closed set as well.

Say $\pi\preim[O]=\bigcup_{k\in\N}F_k$, 
then $O=\bigcup_{k\in\N}\pi[F_k]$ is an $F_\sigma$-set by 
Lemma~\ref{lemma:piclosed}. 
\end{proof}

\begin{remark}
This example is not quite the one that we want because although $\{\infty\}$~is
a $G_\delta$-set it is not a zero-set.
For assume $f:B\cup\{\infty\}\to[0,1]$ is continuous and such 
that $f(\infty)=0$ and $f\bigl(\pi(x,0,0)\bigr)>0$ for all~$x\in\R$.

Then by the method proof of Lemma~\ref{lemma.QtoP} we find, natural 
numbers $k$ and~$m$, an open interval~$J$ and a dense subset~$A$ of~$J$
such that for all $a\in A$ we have $f\bigl(\pi(a,0,0)\bigr)\ge2^{-k}$ 
and $f\bigl(\pi(x,y,0)\bigr)>\frac12f\bigl(\pi(a,0,0)\bigr)$ for 
all $\orpr xy\in U'(a,m)$.
As before this implies that $f\bigl(\pi(x,0,0)\bigr)\ge2^{-k-1}$ 
for all~$x\in J$ and that there is a single $G_\delta$-set~$G$ that is dense
in~$J$ and such that $f\bigl(\pi(x,0,n)\bigr)\ge2^{-k-1}$ for all~$x\in G$ and
all~$n\in\N$.
This would contradict the continuity of~$f$ at~$\infty$.
\end{remark}

\section{A whole Library}
\label{sec.library}

We redo our construction.

We start with the product $N\times\N\times\N$ and add a point~$\infty$.
The basic neighbourhoods of $\infty$ are the sets 
$$
V_k=\bigcup_{m\ge k}\bigcup_{n\ge k} N\times\{\orpr mn\}
$$
We apply the Jones construction to every column in this matrix, so that
we get a book~$B_m$ for every~$m$.

The resulting quotient space, which we call~$L$ is the union of these books
and the point~$\infty$ (yes, $L$~for Library).
We let~$\pi$ denote the quotient map.

\begin{proposition}
The space $L$ is regular, but not completely regular.  
\end{proposition}

\begin{proof}
The space~$L$ is completely regular at all points except possibly~$\infty$.

The space is regular at~$\infty$ by the same argument as for~$B$:
for every~$k$ the closed neighbourhood~$V_{k+1}$ is contained in the interior
of~$V_k$.

To see that $L$ is not completely regular at~$\infty$ let $f:L\to[0,1]$
be continuous and such that $f(x,0,m,0)=0$ for all $x\in\R$ and all $m\in\N$.
By Remark~\ref{rem:Gdelta} applied to every column we obtain a single
dense $G_\delta$-set~$G$ in~$\R$ such that for all $x\in G$
and all $m,n\in\N$ we have $f(\pi(x,0,m,n))=0$.
By continuity this implies that $f(\infty)=0$.  
\end{proof}

\begin{proposition}
Every one-point set in~$L$ is a zero-set.  
\end{proposition}

\begin{proof}
We only need to check this for~$\{\infty\}$.

But this is straightforward: let $f:L\to[0,1]$ take on the value~$2^{-m}$
on book~$B_m$, and let $f(\infty)=0$.  
\end{proof}

To summarize in a theorem.

\begin{theorem}
The space~$L$ is regular, but not completely regular.
Every one-point set is a zero-set of~$L$, and every closed set is 
a $G_\delta$-set.  
\end{theorem}

\begin{proof}
The only thing that still needs to be verified is the last statement.

But that is done almost verbatim as for~$B\cup\{\infty\}$.   
Adapt the proofs of Lemma~\ref{lemma:piclosed} and 
Proposition~\ref{prop:Bperfect} to the present situation.
\end{proof}

\begin{remark}
This example also has an application in lattice theory.
The lattice~$\Open(X)$ of open sets in a topological space is readily 
seen to be complete: the union and the interior of the intersection serve 
as the supremum and infimum, respectively, of a family of open sets.
An important sublattice of~$\Open(X)$ is that of the cozero-sets: $\Coz(X)$.

It turns out that $\Open(X)$ is the Dedekind-MacNeille completion of~$\Coz(X)$
if and only if $X$~is completely regular and each singleton set is a zero-set.

As shown by the present example this result is sharp in that complete
regularity cannot be replaced by regularity.

Thanks to Guram Bezhanishvili for supplying this information.
\end{remark}

\end{document}